\documentclass[authoryear,preprint,12pt]{elsarticle}
\journal{Mathematical Methods of Statistics}

\usepackage{amsmath,amsfonts,amssymb,amsthm,graphicx,hyperref,url}
\usepackage[labelfont=bf]{caption}

\usepackage[inline]{enumitem} % to correctly label and reference the assumptions
\usepackage{xcolor} % to color up the comments
\usepackage{dsfont} % for the \ind command
\usepackage{subcaption} % for subfigures
\usepackage{float} % for better placement of figures
\usepackage{csvsimple} % to import and typeset tables directly from .csv files
\usepackage{booktabs} % to produce professional-quality tables with refined rules
\usepackage{geometry} % to control the margins
\theoremstyle{plain}% Theorem-like structures provided by amsthm.sty
\newtheorem{theorem}{Theorem}

\newtheorem{lemma}{Lemma}

\newtheorem{remark}{Remark}

\newcommand{\N}{\mathbb{N}}

\newcommand{\R}{\mathbb{R}}
\newcommand{\EE}{\mathsf{E}} % Russian style (don't change)
\newcommand{\Var}{\mathsf{Var}} % Russian style (don't change)
\newcommand{\Cov}{\mathsf{Cov}} % Russian style (don't change)
\newcommand{\bb}[1]{\boldsymbol{#1}}
\newcommand{\OO}{\mathcal{O}}
\newcommand{\oo}{\mathrm{o}}
\newcommand{\rd}{\mathrm{d}}
\newcommand{\ind}{\mathds{1}}

\newcommand{\e}{\varepsilon}
\newcommand{\argmax}{\operatorname*{arg\,max}}
\newcommand{\argmin}{\operatorname*{arg\,min}}

\allowdisplaybreaks

\begin{document}

\begin{frontmatter}

\title{Dirichlet kernel density estimation for strongly mixing sequences\\ on the simplex}

\author[a1]{Hanen Daayeb}
\author[a1]{Salah Khardani}
\author[a2]{Fr\'ed\'eric Ouimet\corref{mycorrespondingauthor}}

\address[a1]{D\'epartement de math\'ematiques, Universit\'e de Tunis El Manar, Tunisia}
\address[a2]{D\'epartement de math\'ematiques, Universit\'e du Qu\'ebec \`a Trois-Rivi\`eres, Canada}

\cortext[mycorrespondingauthor]{Corresponding author. Email address: frederic.ouimet2@uqtr.ca}

\begin{abstract}
This paper investigates the theoretical properties of Dirichlet kernel density estimators for compositional data supported on simplices, for the first time addressing scenarios involving time-dependent observations characterized by strong mixing conditions. We establish rigorous results for the asymptotic normality and mean squared error of these estimators, extending previous findings from the independent and identically distributed (iid) context to the more general setting of strongly mixing processes. To demonstrate its practical utility, the estimator is applied to monthly market-share compositions of several Renault vehicle classes over a twelve-year period, with bandwidth selection performed via leave-one-out least squares cross-validation. Our findings underscore the reliability and strength of Dirichlet kernel techniques when applied to temporally dependent compositional data.
\end{abstract}


\end{frontmatter}

\section{Introduction}\label{sec:introduction}

Kernel density estimation is a fundamental task in statistical analysis, aiming to infer the underlying distribution of observed data without imposing restrictive parametric assumptions. While classical multivariate kernel density estimators (KDEs), such as those using Gaussian kernels, have proven effective for data supported on unbounded spaces, as documented extensively in the literature \citep[see, e.g.,][]{MR1319818}, their direct application to compact domains introduces significant boundary bias. This bias arises from the kernel's spillover effect, assigning non-negligible mass outside the domain boundaries, and boundary correction methods can often lead to negative density estimates near the edges of the support.

To address these boundary issues, asymmetric kernels have emerged as powerful tools, inherently adapting their shapes to match the geometry of the boundaries and thus ensuring nonnegativity of density estimates throughout the entire support. Among the univariate asymmetric kernels, beta kernels on $[0,1]$ \citep[e.g.,][]{MR1718494,MR1985506,MR2775207,MR3333996,MR3463548} and gamma kernels on $[0,\infty)$ \citep[e.g.,][]{MR1794247,MR2179543,MR2756423,MR3843043} have attracted particular attention thanks to their intrinsic local adaptivity and excellent boundary performance. Similarly, in the multivariate setting, product kernels on product spaces \citep{MR4415422,MR2568128,MR3384258}, multivariate inverse Gaussian kernels on half-spaces \citep{MR4939549}, Wishart kernels on the cone of positive definite matrices \citep{MR4358612}, and Dirichlet kernels on simplices \citep{doi:10.2307/2347365,MR4319409,MR4544604} extend these boundary-adaptive properties to higher-dimensional supports, carrying over the same bias-mitigating and nonnegativity-preserving behavior despite the greater geometric complexity. A unified theoretical treatment of many such asymmetric kernels is provided by the associated kernel framework; see, for instance, \citet{MR3760293,doi:10.3390/stats4010013}, \citet{MR4859217}, and \citet{doi:10.1080/03610926.2025.2530135}.

Dirichlet KDEs have recently demonstrated notable advantages. Their shape parameters are designed to adapt locally according to the estimation point, significantly reducing boundary bias compared to classical symmetric kernels, and ensure density nonnegativity everywhere within the support \citep{MR4319409}. Furthermore, they have been proven to attain minimax optimal rates under various smoothness conditions on the simplex, highlighting their theoretical efficiency \citep{MR4544604}.

Recent developments have extended the scope of Dirichlet kernel methods beyond density estimation to regression problems on simplices. For instance, \citet{MR4796622} introduced a Dirichlet kernel-based Nadaraya--Watson estimator, analyzing its theoretical properties within the broader framework of conditional U-statistics. \citet{MR4905615} proposed a local linear smoother employing Dirichlet kernels, demonstrating superior performance compared to the Nadaraya--Watson approach. More recently, \citet{arXiv:2502.08461} examined a Dirichlet kernel adaptation of the Gasser--M\"uller estimator, establishing its asymptotic properties and providing a comprehensive performance comparison with both previously mentioned regression methods.

Despite the growing literature on Dirichlet kernel methods for density estimation, existing studies predominantly focus on independent and identically distributed (iid) data. In practice, however, dependence between observations frequently occurs, such as temporal or spatial dependencies, captured through mixing conditions. The theoretical properties of Dirichlet KDEs under dependence scenarios, such as strong mixing conditions, have not yet been explored.

The present paper extends the theory of Dirichlet KDEs by investigating their performance for strongly mixing sequences supported on the simplex. Specifically, we analyze the asymptotic normality and mean squared error (MSE) of these estimators, providing rigorous theoretical results that generalize some of the previous work done by \citet{MR4319409} in the iid context. This study thus fills an essential gap in the statistical literature, paving the way for more robust and practically relevant density estimation methods applicable to a wide array of time-dependent compositional data scenarios.

The paper is organized as follows. Section~\ref{sec:definitions.notations} introduces the necessary definitions and notations, including the Dirichlet kernel and the strong mixing condition considered throughout the paper. Section~\ref{sec:asymptotic.properties} presents the main theoretical results, detailing the MSE and the asymptotic normality of the Dirichlet KDE under strong mixing. Section~\ref{sec:application} provides a real-data illustration, applying the smoothing method to estimate the marginal density of monthly market-share compositions of several Renault vehicle classes evolving over a twelve-year period, with bandwidth selection via leave-one-out Monte Carlo least-squares cross-validation. The proofs of the main theoretical results are given in Section~\ref{sec:proofs}, and technical lemmas supporting those proofs appear in Section~\ref{sec:tech.lemmas}.

\section{Definitions and notations}\label{sec:definitions.notations}

For any integer $d\in \N = \{1,2,\ldots\}$, the $d$-dimensional simplex and its interior are defined by
\[
\mathcal{S}_d = \{\bb{s}\in [0,1]^d: \rVert \bb{s}\rVert_1 \leq 1\}, \qquad \mathrm{Int}(\mathcal{S}_d) = \{\bb{s}\in (0,1)^d: \rVert \bb{s}\rVert_1 < 1\},
\]
where $\lVert \bb{s}\rVert_1 = \sum_{i=1}^d \lvert s_i\rvert$ denotes the $\ell^1$ norm in $\R^d$. For any $u_1,\ldots,u_d,v\in (0,\infty)$, the density of the $\mathrm{Dirichlet}\hspace{0.2mm}(\bb{u},v)$ distribution is given, for every $\bb{s}\in \mathcal{S}_d$, by
\[
K_{\bb{u},v}(\bb{s}) = \frac{\Gamma(\rVert \bb{u}\rVert_1 + v)}{\Gamma(v) \prod_{i=1}^d \Gamma(u_i)} (1 - \rVert \bb{s}\rVert_1)^{v - 1} \prod_{i=1}^d s_i^{u_i - 1}.
\]

Consider a sequence $\bb{X}_1,\ldots,\bb{X}_n$ of $\mathcal{S}_d$-valued random vectors that may be dependent across the index $i$ while each individual vector $\bb{X}_i$ remains, by construction, a composition on the simplex. Each $\bb{X}_i$ has an unknown density $f$, called the target density, with support entirely in $\mathcal{S}_d$. Such sequences arise, for instance, as the first $n$ observations from a stationary stochastic process $(\bb{X}_t)_{t\in \N}$ taking values in $\mathcal{S}_d$, with time dependence typically restricted by conditions like strong mixing. The case where $\bb{X}_1,\ldots,\bb{X}_n$ are independent and identically distributed (iid) is of course included as a special case. Given a bandwidth $b \in (0,\infty)$, the Dirichlet KDE of $f$ at $\bb{s}\in \mathcal{S}_d$ is defined by
\begin{equation}\label{eq:Dirichlet.KDE}
\hat{f}_{n,b}(\bb{s}) = \frac{1}{n} \sum_{i=1}^n \kappa_{\bb{s},b}(\bb{X}_i),
\end{equation}
where, for brevity, $\kappa_{\bb{s},b}(\cdot) = K_{\bb{s}/b + \bb{1}, (1 - \rVert \bb{s}\rVert_1)/b + 1}(\cdot)$ with $\bb{1} = (1,\ldots,1)$ being the $d$-vector of ones.

Here are some notation conventions we will use throughout the rest of the paper. The notation $u = \OO(v)$ means that $\limsup \rvert u/v\rvert \leq C < \infty$ as $b\to 0$ or $n\to \infty$, depending on the context. The positive constant $C$ can depend on the target density $f$ and the dimension $d$, but no other variable unless explicitly written as a subscript. The most common occurrence is a local dependence of the asymptotics with a given point $\bb{s}$ on the simplex, in which case we write $u = \OO_{\bb{s}}(v)$. The notation $u\ll v$ is also used sometimes to mean $u = \OO(v)$ and $u,v\geq 0$, with subscripts again indicating dependence. The notation $u = \oo(v)$ means that $\lim \lvert u/v\rvert = 0$ as $b\to 0$ or $n\to \infty$, and subscripts indicate which parameters the convergence rate can depend on. We use the shorthand $[d] = \{1,\ldots,d\}$ in several places. The bandwidth parameter $b = b(n)$ is always implicitly a function of the number of observations, the only exception being in Section~\ref{sec:tech.lemmas}. To quantify dependence between observations, we use the standard strong mixing coefficient defined by
\[
\alpha(n) = \sup_{k \geq 1} \, \sup_{A \in \mathcal{F}_1^k(\bb{X}), \, B \in \mathcal{F}_{k+n}^{\infty}(\bb{X})} \rvert\mathbb{P}(A \cap B) - \mathbb{P}(A)\mathbb{P}(B)\rvert,
\]
where $\mathcal{F}_i^j(\bb{X})$ denotes the $\sigma$-algebra generated by the random variables $(\bb{X}_i,\ldots,\bb{X}_j)$; see, e.g., \citet[p.~18]{MR1640691}. A sequence $(\bb{X}_i)_{i\in \N}$ is said to be strongly mixing if $\alpha(n) \to 0$ as $n\to \infty$.

\section{Convergence properties}\label{sec:asymptotic.properties}

For each result in this section, the following assumptions will be used:

\begin{enumerate}[label=A\arabic*]\setlength\itemsep{0em}
\item The density $f$ is twice continuously differentiable on $\mathcal{S}_d$. \label{ass:1}
\item For some $p\in (2,\infty)$, the local dependence function, $h_{i,j} = f_{\bb{X}_i,\bb{X}_j} - f_{\bb{X}_i} f_{\bb{X}_j}$, satisfies
\[
\sup_{i\neq j} \rVert h_{i,j}\rVert_p \equiv \sup_{i\neq j} \left\{\int_{\mathcal{S}_d \times \mathcal{S}_d} \lvert h_{i,j}(\bb{x},\bb{x}')\rvert^p \rd \bb{x} \rd \bb{x}'\right\}^{1/p} < \infty.
\] \label{ass:2}
\item For $p\in (2,\infty)$, the mixing coefficient satisfies $\alpha(n) \ll n^{-\nu}$ for some $\nu > 2(p - 1)/(p - 2)$. \label{ass:3}
\item The density $f$ is Lipschitz continuous on $\mathcal{S}_d$. \label{ass:4}
\item The bandwidth $b=b(n)$ satisfies $b\to 0$, $n b^{d/2}\to \infty$ and $n^{1/2} b^{d/4 + 1/2}\to 0$ as $n\to \infty$. \label{ass:5}
\end{enumerate}

\begin{remark}
Assumption~\ref{ass:1} ensures that a second-order Taylor expansion of $f$ is valid in a neighborhood of any interior point, which is indispensable for deriving the leading‐order bias term in the proof of Theorem~\ref{thm:MSE}. Assumptions~\ref{ass:2} and~\ref{ass:3} together control the strength of temporal dependence: the $L^{p}$ bound on the local covariance kernel $h_{i,j}$ and the polynomial decay $\alpha(n)\ll n^{-\nu}$ guarantee that covariance terms are summable and allow blocking arguments to deliver a central limit theorem; see the proof of Theorem~\ref{thm:CLT} for details. Assumption~\ref{ass:4} is minimal for controlling the bias term in the central limit theorem and is automatically satisfied when $f$ has continuous first-order partial derivatives on the simplex. Finally, the bandwidth regime $n b^{d/2}\to \infty$ in Assumption~\ref{ass:5} guarantees that the variance of the estimator at $\bb{s}$ is $\oo_{\bb{s}}(1)$, while the extra condition $n^{1/2}b^{d/4+1/2}\to0$ guarantees that the rescaled bias vanishes in the asymptotic normality result. Taken together, these assumptions are weak enough to cover a wide range of strongly mixing processes commonly encountered with compositional time series.
\end{remark}

First, we establish the asymptotic behavior of the MSE of the Dirichlet KDE.

\begin{theorem}[Mean squared error]\label{thm:MSE}
Suppose that Assumptions~\ref{ass:1}--\ref{ass:3} hold. For any $\bb{s}\in \mathrm{Int}(\mathcal{S}_d)$, we have, as $n\to \infty$,
\[
\begin{aligned}
\mathrm{MSE}[\hat{f}_{n,b}(\bb{s})]
&\equiv \EE\big\{\rvert\hat{f}_{n,b}(\bb{s}) - f(\bb{s})\rvert^2\big\} \\
&= b^2 g^2(\bb{s}) + n^{-1} b^{-d/2} \psi(\bb{s}) f(\bb{s}) + \oo(b^2) + \oo_{\bb{s}}(n^{-1} b^{-d/2}),
\end{aligned}
\]
where
\[
\begin{aligned}
g(\bb{s})
&= \sum_{i\in [d]} (1 - (d + 1)s_i) \frac{\partial}{\partial s_i} f(\bb{s}) + \frac{1}{2} \sum_{i,j\in [d]} s_i (\ind_{\{i = j\}} - s_j) \frac{\partial^2}{\partial s_i \partial s_j} f(\bb{s}), \\
\psi(\bb{s})
&= \frac{(4\pi)^{-d/2}}{(1 - \rVert \bb{s}\rVert_1)^{1/2} \prod_{i\in [d]} s_i^{1/2}}.
\end{aligned}
\]
\end{theorem}

Next, we establish the asymptotic normality of the estimator, characterizing its distributional behavior.

\begin{theorem}[Asymptotic normality and plug-in confidence interval]\label{thm:CLT}
Suppose Assumptions~\ref{ass:2}--\ref{ass:5} hold and assume further that $\nu\geq (d+4)/2$ in Assumption~\ref{ass:3}. For any $\bb{s}\in \mathrm{Int}(\mathcal{S}_d)$,
\[
n^{1/2} b^{d/4} \bigl\{\hat{f}_{n,b}(\bb{s})-f(\bb{s})\bigr\}
\xrightarrow{\mathrm{law}} \mathcal{N}(0,\psi(\bb{s})\,f(\bb{s})),
\qquad n\to\infty.
\]
In particular, a plug-in asymptotic $(1-\alpha)$ confidence interval for $f(\bb{s})$ is
\[
\big[\hat{f}_{n,b}(\bb{s}) \pm \Phi^{-1}(1-\alpha/2) \, \sqrt{\psi(\bb{s}) \, \hat{f}_{n,b}(\bb{s})} \, n^{-1/2} b^{-d/4}\big],
\]
where $\Phi^{-1}$ denotes the quantile function of the standard normal distribution.
\end{theorem}

\section{Real-data illustration}\label{sec:application}

Monthly market-share compositions of five Renault vehicle classes ($A$, $B$, $C$, $D$, $E$) sold in France from January 2003 to August 2015 ($n = 152$) are considered. The dataset is publicly available in the GitHub repository of \citet{BarreiroLaurentThomasAgnan2021} under the name \texttt{BDDSegX.RData}. Writing $\bb{S}_t = (S_{A,t}, S_{B,t}, S_{C,t}, S_{D,t}, S_{E,t}), ~t \in \{1,\ldots,152\}$, the Dirichlet kernel estimator $\hat{f}_{n,b^{\star}}$ in \eqref{eq:Dirichlet.KDE} is applied, where the bandwidth parameter $b^{\star}$ is selected by minimizing the leave-one-out least-squares cross-validation (LSCV), viz.
\[
\mathrm{LSCV}(b)
= \int_{\mathcal{S}_d} \bigl\{\hat{f}_{n,b}(\bb{s})\bigr\}^2 \rd \bb{s}
- 2 \int_{\mathcal{S}_d} \hat{f}_{n,b}(\bb{s}) f(\bb{s}) \rd \bb{s}.
\]
In practice, both terms are approximated by Monte Carlo sampling. Draw $M = 1000$ independent points $\smash{\widetilde{\bb{S}}_1,\ldots,\widetilde{\bb{S}}_M} \sim \mathrm{Uniform}(\mathcal{S}_d) \equiv \mathrm{Dirichlet}(\bb{1},1)$ and approximate the first integral as follows:
\[
\begin{aligned}
\int_{\mathcal{S}_d} \bigl\{\hat{f}_{n,b}(\bb{s})\bigr\}^2 \rd \bb{s}
&\approx \frac{1}{M}\sum_{m=1}^M \bigl\{\hat{f}_{n,b}(\widetilde{\bb{S}}_m)\bigr\}^2 \times \mathrm{Vol}(\mathcal{S}_d) \\
&= \frac{1}{M d!}\sum_{m=1}^M \bigl\{\hat{f}_{n,b}(\widetilde{\bb{S}}_m)\bigr\}^2,
\end{aligned}
\]
since $\mathrm{Vol}(\mathcal{S}_d)=1/d!$. To approximate the second integral, draw random indices $I_1,\ldots,I_M\stackrel{\mathrm{iid}}{\sim}\mathrm{Uniform}\{1,\ldots,n\}$, and compute
\[
\int_{\mathcal{S}_d} \hat{f}_{n,b}(\bb{s}) f(\bb{s}) \rd \bb{s}
\approx \frac{1}{M} \sum_{m=1}^M \hat{f}_{n,b}^{(-I_m)}(\bb{X}_{I_m}),
\]
where $\hat{f}_{n,b}^{(-i)}$ denotes the Dirichlet KDE computed without the $i$th observation. The resulting Monte Carlo approximation of the LSCV criterion is
\[
\mathrm{LSCV}_{\mathrm{MC}}(b)
= \frac{1}{M d!}\sum_{m=1}^M \bigl\{\hat{f}_{n,b}(\widetilde{\bb{S}}_m)\bigr\}^2
- \frac{2}{M} \sum_{m=1}^M \hat{f}_{n,b}^{(-I_m)}(\bb{X}_{I_m}),
\]
which is evaluated over the grid $b_1,\ldots,b_K$ (say, $0.01,\ldots,0.50$) to yield the chosen bandwidth:
\[
b^{\star} = \argmin_{1\leq k\leq K} ~\mathrm{LSCV}_{\mathrm{MC}}(b_k).
\]

For exploratory analysis, each unordered pair $(S_i,S_j)$ is combined with the residual component $1 - S_i - S_j$, producing ten trivariate series that are smoothed with their own bandwidths $b^{\star}_{i,j}$, as illustrated in Figure~\ref{fig:density.plots}.

\begin{figure}[H]
\centering
\setlength{\tabcolsep}{2pt}
\renewcommand{\arraystretch}{0}
\begin{tabular}{ccc}
 % Row 1
 \begin{subfigure}{0.31\textwidth}
 \includegraphics[width=\linewidth]{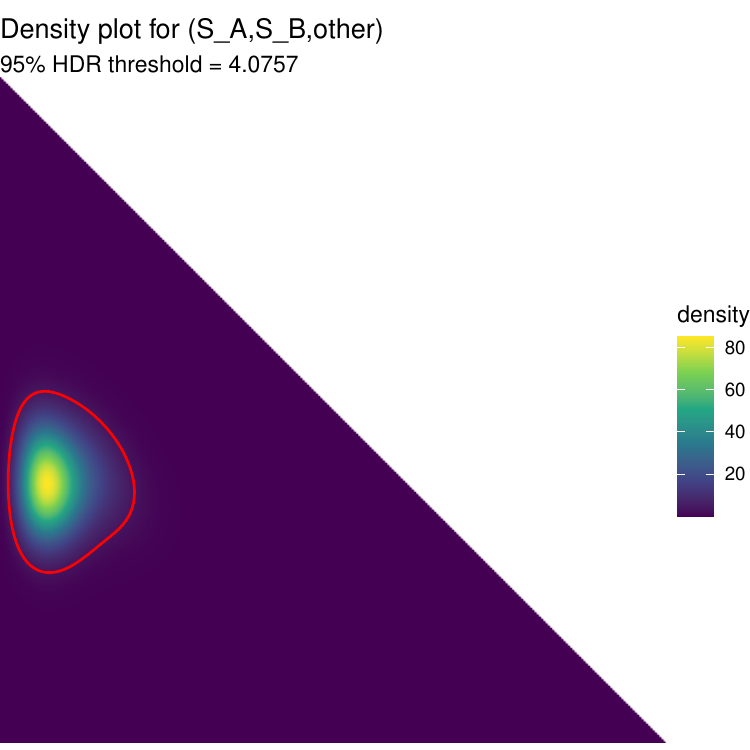}
 \end{subfigure} &
 \begin{subfigure}{0.31\textwidth}
 \includegraphics[width=\linewidth]{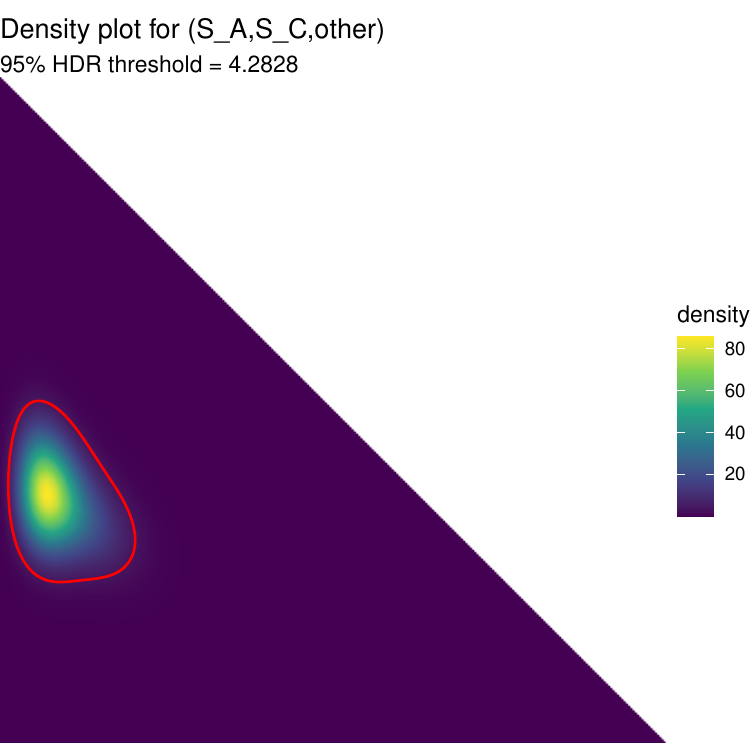}
 \end{subfigure} &
 \begin{subfigure}{0.31\textwidth}
 \includegraphics[width=\linewidth]{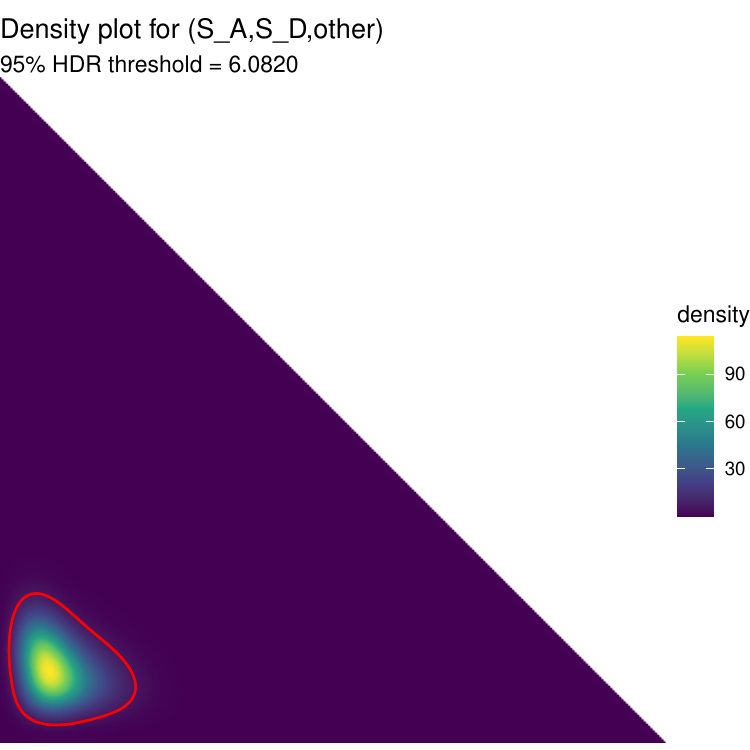}
 \end{subfigure} \\[0pt]
 % Row 2
 \begin{subfigure}{0.31\textwidth}
 \includegraphics[width=\linewidth]{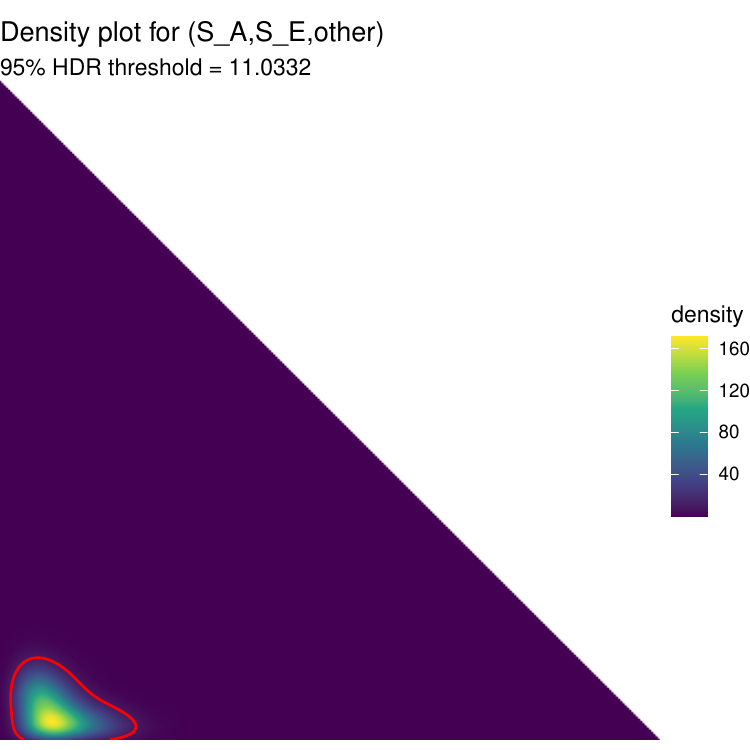}
 \end{subfigure} &
 \begin{subfigure}{0.31\textwidth}
 \includegraphics[width=\linewidth]{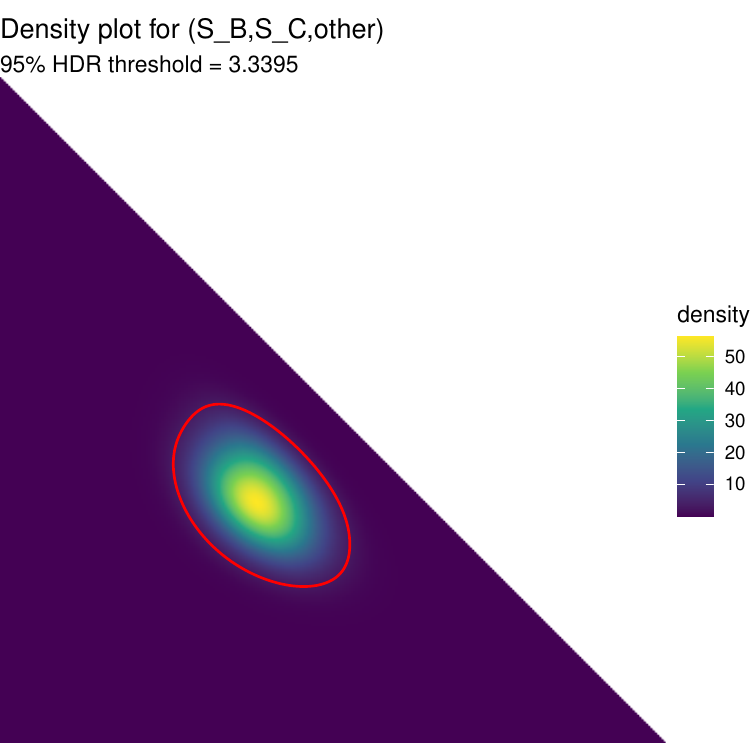}
 \end{subfigure} &
 \begin{subfigure}{0.31\textwidth}
 \includegraphics[width=\linewidth]{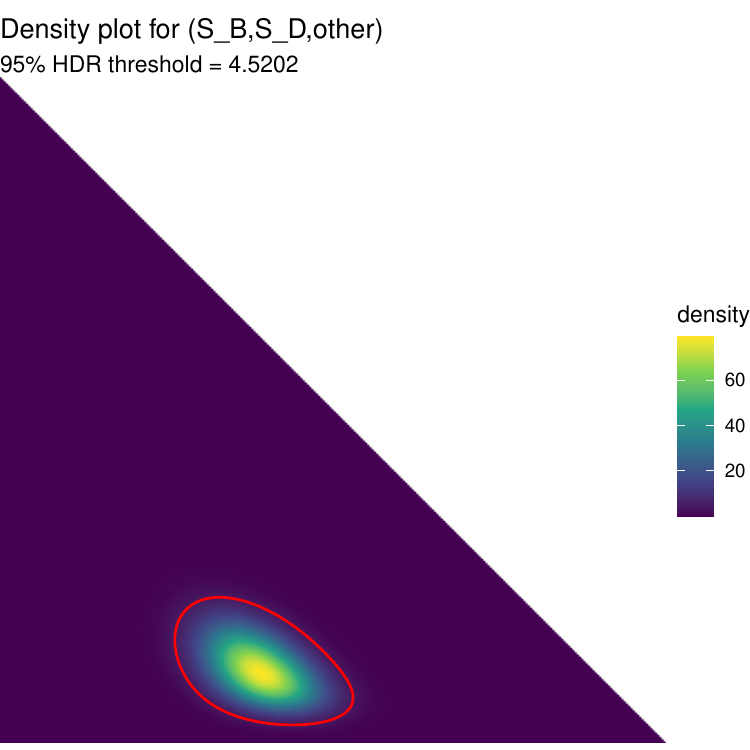}
 \end{subfigure} \\[0pt]
 % Row 3
 \begin{subfigure}{0.31\textwidth}
 \includegraphics[width=\linewidth]{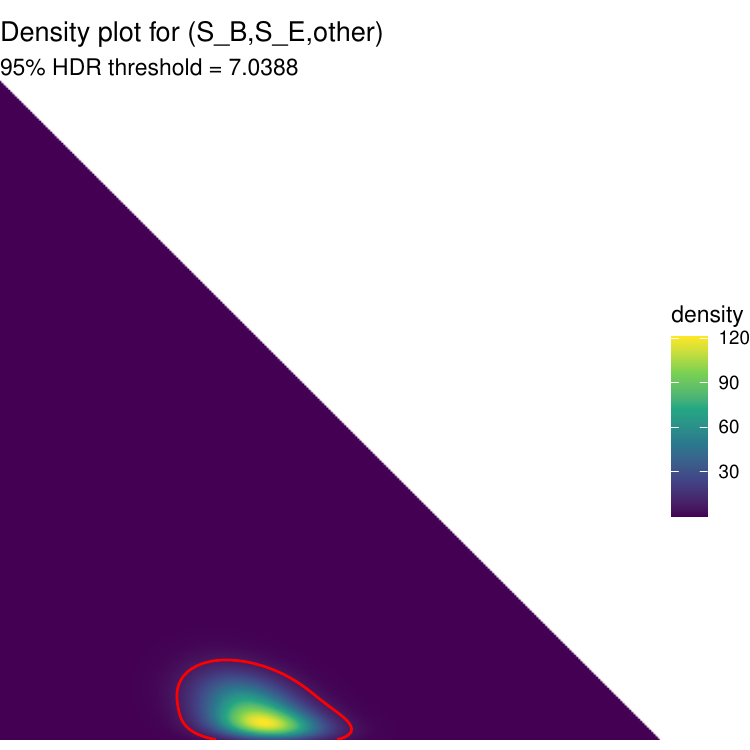}
 \end{subfigure} &
 \begin{subfigure}{0.31\textwidth}
 \includegraphics[width=\linewidth]{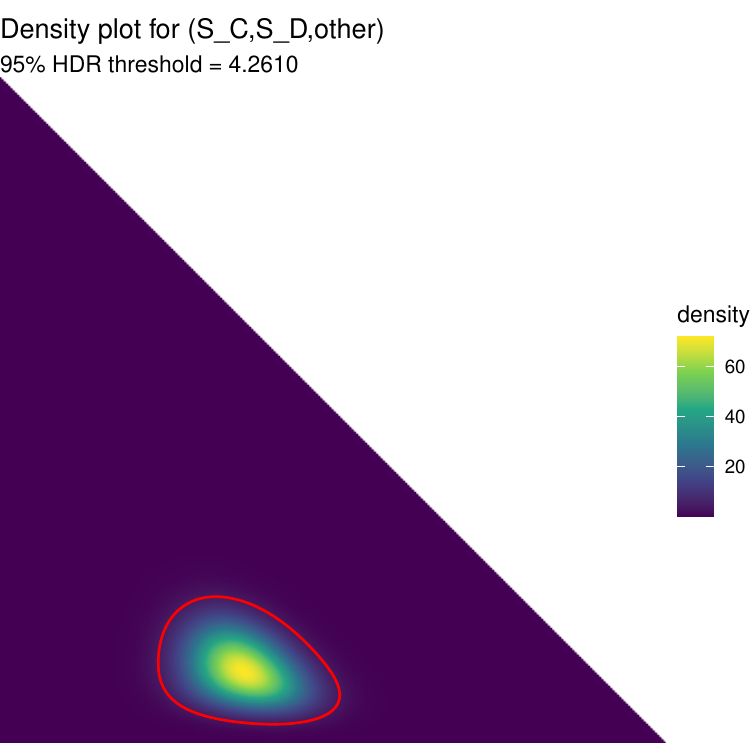}
 \end{subfigure} &
 \begin{subfigure}{0.31\textwidth}
 \includegraphics[width=\linewidth]{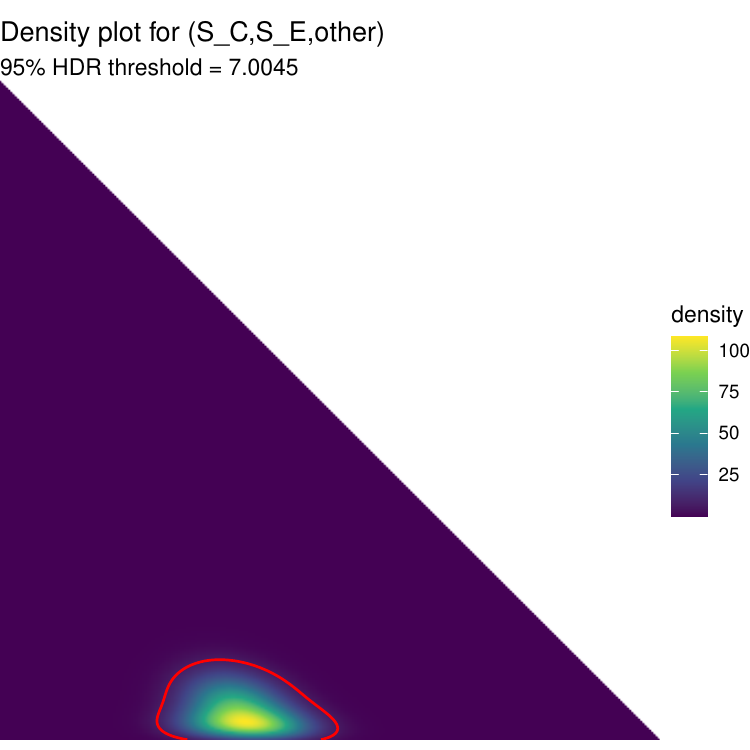}
 \end{subfigure} \\[-10pt]
 % Row 4: single center subfigure
 &
 \begin{subfigure}{0.31\textwidth}
 \includegraphics[width=\linewidth]{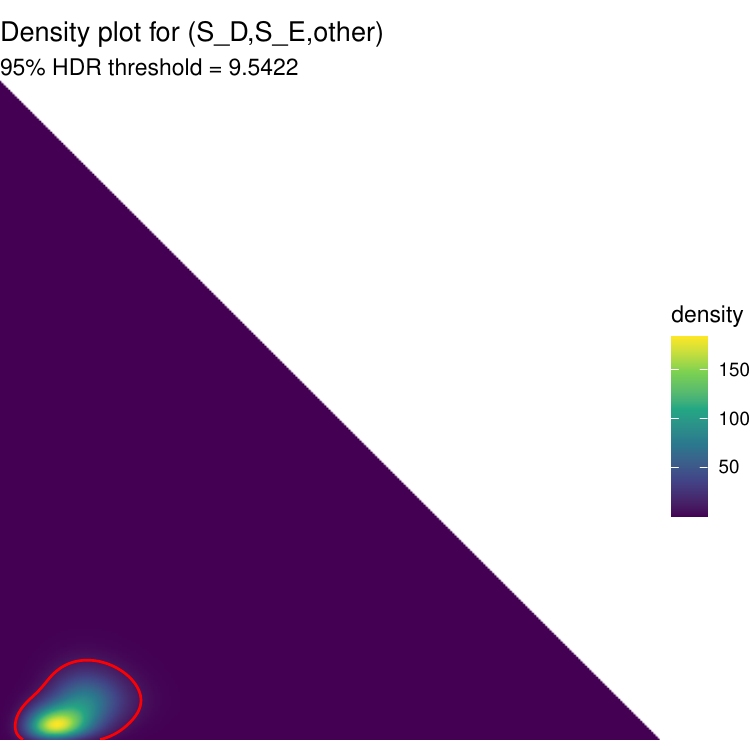}
 \end{subfigure} &
\end{tabular}
\caption{Kernel density estimates for each three-component composition $(S_i,S_j,\text{other})$ on the simplex, displayed with a common color scale (darker shades represent lower density). The red closed curve marks the boundary of the $95\%$ highest-density region (HDR).}
\label{fig:density.plots}
\end{figure}

The resulting density level at which the estimator encloses exactly $95\,\%$ of its mass --- i.e., the threshold that defines the highest-density region (HDR) for each pair --- is reported in Table~\ref{tab:hdr-thresholds}. These thresholds summarize, for every pair of Renault vehicle classes, how high the estimated density must be for a point on the simplex to belong to the long-term $95\,\%$ HDR. These thresholds can inform Renault’s strategic decision-making. By pinpointing the highest-density regions of joint market shares, the marketing team gains quantitative benchmarks for setting realistic cross-segment share targets, spotting unusual deviations early, and reallocating budgets across the classes $A$--$E$ to optimize their model portfolio.

%\vspace{-3mm}
%\input{application/HDR_thresholds_table.tex}

\begin{table}[ht]
\centering
\begin{tabular}{lc}
\toprule
Composition & HDR threshold \\
\midrule
$(S_A,S_B,\text{other})$ & 4.08 \\
$(S_A,S_C,\text{other})$ & 4.28 \\
$(S_A,S_D,\text{other})$ & 6.08 \\
$(S_A,S_E,\text{other})$ & 11.03 \\
$(S_B,S_C,\text{other})$ & 3.34 \\
$(S_B,S_D,\text{other})$ & 4.52 \\
$(S_B,S_E,\text{other})$ & 7.04 \\
$(S_C,S_D,\text{other})$ & 4.26 \\
$(S_C,S_E,\text{other})$ & 7.00 \\
$(S_D,S_E,\text{other})$ & 9.54 \\
\bottomrule
\\
\end{tabular}
\caption{Estimated 95\% highest-density-region thresholds for each composition $(S_i,S_j,\text{other})$ of Renault vehicle shares.}
\label{tab:hdr-thresholds}
\end{table}

\section{Discussion}\label{sec:discussion}

The analysis carried out in this paper establishes a precise large-sample description of the Dirichlet KDE on the simplex when the observations form a strongly mixing sequence. The MSE expansion in Theorem~\ref{thm:MSE} shows that the squared bias is of order $b^{2}$ while the variance is of order $n^{-1}b^{-d/2}$, as in the iid case. The asymptotic normality result in Theorem~\ref{thm:CLT} demonstrates that the centered and rescaled estimator converges in distribution to a Gaussian limit whose variance coincides with the variance term in the MSE. The real-data illustration in Section~\ref{sec:application} confirms that the method remains operative in practical situations involving temporal dependence. It was used in particular to compute long-term 95\% highest-density regions for every three-component composition of Renault vehicle classes, using the monthly sales data, thereby showing where the corresponding market-share compositions were most concentrated in France from 2003 to 2015.

The Dirichlet kernel estimator analyzed in this paper is defined with respect to the Euclidean geometry inherited from the ambient space $\R^d$, restricted to the simplex. A natural alternative is the Aitchison geometry, which is specifically tailored for compositional data. In Aitchison geometry, distances are computed after applying the isometric log-ratio (ilr) transformation, which maps the simplex to a Euclidean space where standard operations become valid; see, e.g., \citet{MR3328965}. A related transformation is the centered log-ratio (clr), which linearizes the simplex into a constrained hyperplane in $\R^{d+1}$, and under which the Aitchison distance corresponds to the Euclidean distance in that subspace. One possible research direction is to transport the Dirichlet kernel framework through a log-ratio transformation, then derive the large-sample properties of the back-transformed estimator on the simplex. The strong mixing coefficients used here would then pertain to the transformed process, but the regularity conditions on the density would need to be re-expressed in terms of the Aitchison norm or corresponding coordinates. The benefit would be a density estimator that is equivariant under perturbation and powering, two key operations in compositional data analysis. However, the tools required for such an extension significantly complicate the theoretical development, so this lies beyond the scope of the present paper.

Another promising avenue is the incorporation of spatial dependence. In many environmental or geological applications, compositions are observed across a spatial grid and exhibit correlation that decays with distance instead of time. Specifically, strong mixing for random fields stipulates that the maximal dependence between the $\sigma$-fields of any finite sets $S,T$ vanishes as $\operatorname{dist}(S,T)\to\infty$. Extending the proofs would therefore require adapting the blocking argument in Section~\ref{sec:proofs} to multi-index collections and verifying that the $L^{p}$-based dependence conditions remain compatible with the spatial mixing rate. The bias analysis would be unaffected, but the variance and covariance terms would involve lattice sums that depend on the dimension of the spatial index set. Deriving optimal bandwidth rates in that setting could reveal a different balance between bias and variance than in the one-dimensional temporal case.

Finally, the question of strong uniform convergence remains open. The asymptotic normality result guarantees pointwise stochastic fluctuations, yet many practical tasks, such as mode estimation \citep[e.g.,][]{MR2780217,MR2817352,MR2887564,MR3474998} or bump hunting (i.e., locating and assessing the statistical significance of local maxima of an unknown density function), require uniform strong consistency of the estimator and its first-order partial derivatives. Under additional smoothness of the target density and a bandwidth sequence that satisfies $n b^{d/2} / \log n \to \infty$, almost-sure uniform convergence might hold on $\mathcal{S}_d$ or an increasing sequence of compacts that fills up $\mathcal{S}_d$; cf.\  Theorem~4 of \citet{MR4319409} in the iid setting. Assume, for instance, that $\smash{\sup_{\bb{t}\in\mathcal{S}_{d}}\lvert \hat{f}_{n,b}(\bb{t})-f(\bb{t})\rvert\xrightarrow{\text{a.s.}} 0}$. Define the mode estimator $\hat{\bb{s}}_{n} = \argmax_{\bb{t}\in\mathcal{S}_{d}}\hat{f}_{n,b}(\bb{t})$ and let $\bb{s}_{0}$ be the unique maximizer of $f$. It follows that
\begin{equation}\label{eq:unif.bound}
\lvert f(\hat{\bb{s}}_{n})-f(\bb{s}_{0})\rvert
\leq \lvert f(\hat{\bb{s}}_{n})-\hat{f}_{n,b}(\hat{\bb{s}}_{n})\rvert ~+~ \rvert\hat{f}_{n,b}(\hat{\bb{s}}_{n})-f(\bb{s}_{0})\rvert
\leq 2 \sup_{\bb{t}\in\mathcal{S}_{d}}\lvert \hat{f}_{n,b}(\bb{t})-f(\bb{t})\rvert.
\end{equation}
A second-order Taylor expansion of $f$ around $\bb{s}_{0}$ yields
\begin{equation}\label{eq:taylor.mode}
f(\hat{\bb{s}}_{n})-f(\bb{s}_{0}) = \frac{1}{2} (\hat{\bb{s}}_{n}-\bb{s}_{0})^{\top}\nabla^{2}f(\bb{s}_{n}^{\star})(\hat{\bb{s}}_{n}-\bb{s}_{0}),
\end{equation}
with $\bb{s}_{n}^{\star}$ lying on the segment between $\bb{s}_{0}$ and $\hat{\bb{s}}_{n}$. Combining \eqref{eq:unif.bound} and \eqref{eq:taylor.mode}, and using the smallest eigenvalue $\lambda_{\min}(-\nabla^{2}f(\bb{s}_{n}^{\star}))>0$, gives
\[
\begin{aligned}
\lVert \hat{\bb{s}}_{n}-\bb{s}_{0}\rVert_2^2
&= \frac{\lVert \hat{\bb{s}}_{n}-\bb{s}_{0}\rVert_2^2}{\lvert \frac{1}{2} (\hat{\bb{s}}_{n}-\bb{s}_{0})^{\top}\nabla^{2}f(\bb{s}_{n}^{\star})(\hat{\bb{s}}_{n}-\bb{s}_{0})\rvert} \lvert f(\hat{\bb{s}}_{n})-f(\bb{s}_{0})\rvert \\
&\leq 4 \, \frac{\sup_{\bb{t}\in\mathcal{S}_{d}}\rvert\hat{f}_{n,b}(\bb{t})-f(\bb{t})\rvert}{\lambda_{\min}(-\nabla^{2}f(\bb{s}_{n}^{\star}))} \xrightarrow{\text{a.s.}} 0.
\end{aligned}
\]
Hence, the uniform strong consistency of the estimator would translate into $\hat{\bb{s}}_{n}\xrightarrow{\mathrm{a.s.}}\bb{s}_{0}$, and could also open the door more generally to rigorous inference for features such as modes, ridges, and highest-density regions on the simplex.

\section{Proofs}\label{sec:proofs}

\subsection{Proof of Theorem~\ref{thm:MSE}}

Following \citet[p.~44]{MR1640691}, we have the decomposition
\[
\EE\big\{\lvert \hat{f}_{n,b}(\bb{s}) - f(\bb{s})\rvert^2\big\}
= \big[\EE\{\hat{f}_{n,b}(\bb{s})\}- f(\bb{s})\big]^2 + \frac{1}{n} \Var\{\kappa_{\bb{s},b}(\bb{X}_1)\} + \frac{1}{n^2} \left(\mathcal{C}_{1,n} + \mathcal{C}_{2,n}\right),
\]
where
\[
\mathcal{C}_{1,n} = \sum_{1\leq \rvert i-j\rvert\leq \beta_n} \Cov\{\kappa_{\bb{s},b}(\bb{X}_i),\kappa_{\bb{s},b}(\bb{X}_j)\}, \quad
\mathcal{C}_{2,n} = \sum_{\beta_n + 1\leq \lvert i-j\rvert\leq n - 1} \Cov\{\kappa_{\bb{s},b}(\bb{X}_i),\kappa_{\bb{s},b}(\bb{X}_j)\},
\]
for some separator $\beta_n \leq n-1$ to be chosen later. The asymptotics of the squared bias and variance were derived under Assumption~\ref{ass:1} in \citet[Theorem~1]{MR4319409}, so it is sufficient to prove that $\mathcal{C}_{1,n}$ and $\mathcal{C}_{2,n}$ are $\oo_{\bb{s}}(n b^{-d/2})$.

Let $p\in (2,\infty)$ and $q\in (1,2)$ be given such that $1/p + 1/q = 1$. On the one hand, applying H\"older's inequality, followed by Assumption~\ref{ass:2} and the estimate on the $L^q$ norm of the Dirichlet kernel in Lemma~\ref{lem:kappa.L.q.norm} of Section~\ref{sec:tech.lemmas}, we obtain, for all $i\neq j$,
\[
\begin{aligned}
\sup_{i\neq j} \rvert\Cov\{\kappa_{\bb{s},b}(\bb{X}_i),\kappa_{\bb{s},b}(\bb{X}_j)\}\rvert
&= \sup_{i\neq j} \left\rvert \int_{\mathcal{S}_d} \kappa_{\bb{s},b}(\bb{x}) \kappa_{\bb{s},b}(\bb{x}') h_{i,j}(\bb{x},\bb{x}') \rd \bb{x} \rd \bb{x}'\right\rvert \\
&\leq \rVert \kappa_{\bb{s},b}\rVert_q^2 \, \sup_{i\neq j} \rVert h_{i,j}\rVert_p \\
&\ll \frac{b^{-d/p} \psi^{2/p}(\bb{s})}{2^{-d/p}q^{d/q}}.
\end{aligned}
\]
This last bound shows that
\[
\begin{aligned}
\lvert \mathcal{C}_{1,n}\rvert
&\ll n \beta_n \times b^{-d/p} \psi^{2/p}(\bb{s}) \\
&\ll_{\bb{s}} n b^{-d/p} \beta_n.
\end{aligned}
\]
On the other hand, using Billingsley's inequality \citep[Corollary~1.1]{MR1640691} in conjunction with the upper bound on the supremum norm of the Dirichlet kernel in Lemma~\ref{lem:local.bound}, we have, for all $i\neq j$,
\[
\begin{aligned}
\rvert\Cov\{\kappa_{\bb{s},b}(\bb{X}_i),\kappa_{\bb{s},b}(\bb{X}_j)\}\rvert
&\leq 4 \, \lVert \kappa_{\bb{s},b}\rVert_\infty^2 \alpha(\lvert i-j\rvert) \\
&\ll_{\bb{s}} b^{-d} \alpha(\rvert i-j\rvert).
\end{aligned}
\]
Together with the condition $\alpha(n) \ll n^{-\nu}$ in Assumption~\ref{ass:3}, this last bound shows that
\[
\begin{aligned}
\lvert \mathcal{C}_{2,n}\rvert
&\ll_{\bb{s}} b^{-d} \sum_{\beta_n + 1\leq \lvert i-j\rvert\leq n - 1} \alpha(|i - j|) \\
&\ll n b^{-d} \int_{\beta_n}^{\infty} x^{-\nu} \rd x \\
&\ll n b^{-d} \beta_n^{-\nu + 1}.
\end{aligned}
\]
Since $\nu>2(p - 1)/(p - 2)$ in Assumption~\ref{ass:3}, choosing $\beta_n = \min\{b^{-d/(\nu q)},n-1\}$ yields
\[
\mathcal{C}_{1,n} + \mathcal{C}_{2,n} = \OO_{\bb{s}}\big[n b^{-d \{1/p + 1/(\nu q)\}}\big] = \oo_{\bb{s}}(n b^{-d/2}).
\]
This concludes the proof.

\subsection{Proof of Theorem~\ref{thm:CLT}}

Let $\bb{s}\in \mathrm{Int}(\mathcal{S}_d)$ be given, and consider the decomposition
\begin{equation}\label{eq:decomp}
n^{1/2} b^{d/4} \{\hat{f}_{n,b}(\bb{s}) - f(\bb{s})\}
= n^{1/2} b^{d/4} \big[\hat{f}_{n,b}(\bb{s}) - \EE\{\hat{f}_{n,b}(\bb{s})\}\big] + n^{1/2} b^{d/4} \big[\EE\{\hat{f}_{n,b}(\bb{s})\} - f(\bb{s})\big].
\end{equation}
The second term on the right-hand side of \eqref{eq:decomp} is $\OO(n^{1/2} b^{d/4 + 1/2})$ under Assumption~\ref{ass:4} by Equation~(12) of \citet{MR4319409}, which, in turn, is $\oo(1)$ by Assumption~\ref{ass:5}.

It remains to show that the first term on the right-hand side of \eqref{eq:decomp} is asymptotically normal. One has
\begin{equation}\label{sum1}
n^{1/2} b^{d/4} \big[\hat{f}_{n,b}(\bb{s}) - \EE\{\hat{f}_{n,b}(\bb{s})\}\big]
= n^{-1/2} b^{d/4} \sum_{i=1}^n \Xi_{i,n}(\bb{s})
\end{equation}
where $\Xi_{i,n}(\bb{s})= \kappa_{\bb{s},b}(\bb{X}_i)-\EE\{\kappa_{\bb{s},b}(\bb{X}_i)\}$.

To handle strongly mixing random variables (under Assumption~\ref{ass:2}), we use the well-known sectioning device introduced by \citet[pp.~228--232]{MR58896}. One first selects positive integer sequences $(p_n)_n$ and $(q_n)_n$ diverging to infinity as $n\to \infty$, with $k_n = \left\lfloor n/(p_n + q_n)\right\rfloor\to \infty$, and any $\nu > \max\{2(p - 1)/(p - 2),(d+4)/2\}$ for which Assumption~\ref{ass:3} holds, such that, as $n\to \infty$,
\begin{center}
\begin{enumerate*}[label=(\alph*)]\setlength\itemsep{0em}
\item $\displaystyle \frac{k_n q_n}{n}\to 0$;\quad \label{ass:6.a}
\item $\displaystyle \frac{p_n}{p_n + q_n}\to 1$;\quad \label{ass:6.b}
\item $\displaystyle \frac{k_n q_n^2 n^{-1} b^{-d/2}}{(p_n + q_n)^{\nu}}\to 0$;\quad \label{ass:6.c}
\item $\displaystyle \frac{k_n}{q_n^{\nu}}\to 0$;\quad \label{ass:6.d}
\item $\displaystyle \frac{p_n}{n^{1/2} b^{d/4}}\to 0$. \label{ass:6.e}
\end{enumerate*}
\end{center}

\noindent
For example, one convenient specification is to take the bandwidth as $b = n^{-2/(d+4)}$, choose the big-block length $p_n = n^{-\e + 2/(d+4)}$, and the small-block length $q_n = n^{1/\nu}$; this selection works for any fixed $\e\in (0,2/(d+4) - 1/\nu)$.

One then splits the set of indices $\{1, \ldots,n\}$ of the sum \eqref{sum1} into $(2k_n + 1)$ subsets with $k_n$ big blocks of size $p_n$ and $k_n$ small blocks of size $q_n$. More specifically, for $j \in \{ 1, \ldots, k_n\}$, let
\[
\begin{aligned}
&\mathcal{I}_j = \{(j - 1)(p_n + q_n) + 1,\ldots,(j - 1)(p_n + q_n) + p_n \}, \\
&\mathcal{J}_j = \{(j - 1)(p_n + q_n) + p_n + 1,\ldots,j(p_n + q_n)\},
\end{aligned}
\]
be the $j$th big and small block, respectively, and let the remaining indices form the set $\{k_n (p_n + q_n) + 1, \ldots, n\}$, which may be empty. Next, define the following random variables for $j\in \{ 1, \ldots, k_n\}$:
\[
U_j(\bb{s}) = \sum_{i\in\mathcal{I}_j}\Xi_{i,n}(\bb{s}), \qquad
V_j(\bb{s}) = \sum_{i\in\mathcal{J}_j}\Xi_{i,n}(\bb{s}).
\]
It follows that
\[
\begin{aligned}
\sum_{i=1}^n \Xi_{i,n}(\bb{s})
&= \sum_{j=1}^{k_n} U_j(\bb{s}) + \sum_{j=1}^{k_n}V_j(\bb{s}) + \sum_{i=k_n (p_n + q_n) + 1}^n \Xi_{i,n}(\bb{s}) \\
&\equiv S_{1,n}(\bb{s}) + S_{2,n}(\bb{s}) + S_{3,n}(\bb{s}).
\end{aligned}
\]

\newpage
\noindent
To conclude, it suffices to show that, as $n\to \infty$,
\begin{enumerate}[label=(\roman*)]\setlength\itemsep{0em}
\item $\displaystyle n^{-1/2} b^{d/4} S_{2,n}(\bb{s}) \stackrel{L^2}{\longrightarrow} 0$;\quad \label{proof:i}
\item $\displaystyle n^{-1/2} b^{d/4} S_{3,n}(\bb{s}) \stackrel{L^2}{\longrightarrow} 0$;\quad \label{proof:ii}
\item $\displaystyle n^{-1/2} b^{d/4} S_{1,n}(\bb{s}) \stackrel{\mathrm{law}}{\longrightarrow} \mathcal{N}(0, \psi(\bb{s}) f(\bb{s}))$. \label{proof:iii}
\end{enumerate}

\noindent
{\it Proof of \ref{proof:i}:}
First, observe that
\[
\begin{aligned}
\EE\{S_{2,n}(\bb{s})^2\}
&= \sum_{j=1}^{k_n} \EE\{V_j(\bb{s})^2\} + 2 \sum_{1 \leq i < j \leq k_n} \EE\{V_i(\bb{s}) V_j(\bb{s})\} \\
&= k_n q_n \EE\{\Xi_{1,n}^2(\bb{s})\} + 2 k_n \sum_{1\leq i < j \leq q_n} \EE\{\Xi_{i,n}(\bb{s})\Xi_{j,n}(\bb{s})\}
+ 2 \sum_{1\leq i < j \leq k_n} \EE\{V_i(\bb{s}) V_j(\bb{s})\} \\[-1mm]
&\equiv \mathcal{T}_{1,n} + \mathcal{T}_{2,n} + \mathcal{T}_{3,n}.
\end{aligned}
\]
Using Lemma~\ref{lem:kappa.L.q.norm} with $q=2$, note that
\[
\mathcal{T}_{1,n} \ll_{\bb{s}} k_n q_n b^{-d/2} \stackrel{\ref{ass:6.a}}{=} \oo_{\bb{s}}(n b^{-d/2}).
\]
Under Assumptions~\ref{ass:2}--\ref{ass:3}, the term $\mathcal{T}_{2,n}$ is asymptotically negligible compared to $\mathcal{T}_{1,n}$ for the same reason that the covariance terms in the proof of Theorem~\ref{thm:CLT} were negligible in front of the variance term.
Next, using \citet[p.~6, Formula~1.12a]{MR3642873} followed by the local bound on the Dirichlet kernel in Lemma~\ref{lem:local.bound}, we get
\[
\begin{aligned}
\rvert\EE\{V_i(\bb{s}) V_j(\bb{s})\}\rvert
&\leq 2 \rVert V_i(\bb{s})\rVert_{\infty} \rVert V_j(\bb{s})\rVert_{\infty} \alpha(p_n + (j - i - 1)(p_n + q_n)) \\
&\ll_{\bb{s}} (q_n b^{-d/2})^2 \alpha(p_n + (j - i - 1)(p_n + q_n)).
\end{aligned}
\]
In turn, under Assumption~\ref{ass:3}, the above implies
\[
\begin{aligned}
\mathcal{T}_{3,n}
&\ll_{\bb{s}} (q_n b^{-d/2})^2 \sum_{i=1}^{k_n - 1} \sum_{\ell=0}^{k_n - i - 1} \{p_n + \ell (p_n + q_n)\}^{-\nu} \\
&\stackrel{\ref{ass:6.b}}{\ll} \frac{(q_n b^{-d/2})^2}{(p_n + q_n)^{\nu}} \sum_{i=1}^{k_n - 1}\sum_{\ell=0}^{\infty} (1/2 + \ell)^{-\nu} \\
&\ll \frac{k_n (q_n b^{-d/2})^2}{(p_n + q_n)^{\nu}},
\end{aligned}
\]
so that
\[
\mathcal{T}_{3,n} \stackrel{\ref{ass:6.c}}{=} \oo_{\bb{s}}(n b^{-d/2}).
\]
Putting the above estimates together proves that \ref{proof:i} holds.

\bigskip
\noindent
{\it Proof of \ref{proof:ii}:}
Given that $S_{3,n}(\bb{s})$ is a small $q_n$-block truncated by the sample size $n$, the proof is analogous to that of \ref{proof:i} but easier since the block has size at most $q_n$ instead of size equal to $q_n$. The details are omitted for conciseness.

\bigskip
\noindent
{\it Proof of \ref{proof:iii}:}
Using Lemma~1.1 of \citet{MR121856}, followed by Assumption~\ref{ass:3} and \ref{ass:6.d}, the characteristic function of $S_{1,n}(\bb{s})$ satisfies
\[
\begin{aligned}
\bigg\lvert \EE\bigg[\exp\bigg\{i t n^{-1/2} b^{d/4} S_{1,n}(\bb{s})\bigg\}\bigg] - \prod_{j=1}^{k_n} \EE\bigg[\exp\{i t n^{-1/2} b^{d/4} U_j(\bb{s})\}\bigg]\bigg\rvert
&\leq 16(k_n - 1) \alpha(q_n) \\
&\ll k_n q_n^{-\nu}\to 0.
\end{aligned}
\]
Hence, the $U_j$'s are asymptotically independent inside $S_{1,n}(\bb{s})$. Moreover, using the local bound in Lemma~\ref{lem:local.bound}, and \ref{ass:6.e}, we have
\[
n^{-1/2} b^{d/4} \lvert U_j(\bb{s})\rvert \ll_{\bb{s}} n^{-1/2} b^{-d/4} p_n\to 0.
\]
It follows that $S_{1,n}(\bb{s})$ satisfies Lindeberg's condition,
\[
\frac{1}{(n^{-1/2} b^{d/4})^2} \sum_{j=1}^{k_n} \EE\left[U_j(\bb{s})^2 \ind_{\{U_j(\bb{s}) > \e n^{-1/2} b^{d/4} \sqrt{\psi(\bb{s}) f(\bb{s})}\}}\right]\to 0, \quad \e > 0,
\]
given that for $n$ large enough, the indicator becomes identically $0$. This proves \ref{proof:iii} and completes the proof of the theorem.

\section{Technical lemmas}\label{sec:tech.lemmas}

The first lemma studies the asymptotics of the $L^q$ norm of the Dirichlet kernel for all $q\in (1,\infty)$.

\begin{lemma}\label{lem:kappa.L.q.norm}
Let $q\in (1,\infty)$ and $\bb{s}\in \mathrm{Int}(\mathcal{S}_d)$ be given. One has, as $b\to 0$,
\[
\rVert \kappa_{\bb{s},b}\rVert_q^2 = \frac{b^{-d/p} \psi^{2/p}(\bb{s})}{2^{-d/p} q^{d/q}} \{1 + \OO_{\bb{s}}(b)\}.
\]
\end{lemma}

\begin{proof}
Define
\[
\begin{aligned}
A_{b,q}(\bb{s})
&= \frac{\kappa_{\bb{s},b}^q(\cdot)}{K_{q\bb{s}/b + \bb{1}, q(1 - \rVert \bb{s}\rVert_1)/b + 1}(\cdot)} \\
&= \frac{\Gamma\{q(1 - \rVert \bb{s}\rVert_1)/b + 1\} \prod_{i\in [d]} \Gamma(q s_i/b + 1)}{\Gamma^q\{(1 - \rVert \bb{s}\rVert_1)/b + 1\} \prod_{i\in [d]} \Gamma^q(s_i/b + 1)} \times \frac{\Gamma^q(1/b + d + 1)}{\Gamma(q/b + d + 1)}.
\end{aligned}
\]
If we denote $R(z) = \sqrt{2\pi} e^{-z} z^{z + 1/2}/\Gamma(z + 1)$ for arbitrary $z\geq 0$, and
\[
S_{b,q}(\bb{s}) = \frac{R^q\{(1 - \lVert \bb{s}\rVert_1)/b\} \prod_{i\in [d]} R^q(s_i/b)}{R\{q(1 - \rVert \bb{s}\rVert_1)/b\} \prod_{i\in [d]} R(q s_i/b)} \times \frac{R(q/b + d)}{R^q(1/b + d)},
\]
then we have
\[
\begin{aligned}
A_{b,q}(\bb{s})
&= \frac{q^{q(1 - \rVert \bb{s}\rVert_1)/b + 1/2} \prod_{i\in [d]} q^{qs_i/b + 1/2}}{(2\pi)^{(q - 1)d/2} \{(1 - \rVert \bb{s}\rVert_1)/b\}^{(q - 1)/2} \prod_{i\in [d]} (s_i/b)^{(q - 1)/2}} \times \frac{e^{-(q - 1)d} (1/b + d)^{q/b + qd + q/2}}{(q/b + d)^{q/b + d + 1/2}} \times S_{b,q}(\bb{s}) \notag \\
&= \frac{b^{(q - 1) (d + 1)/ 2} (1/b + d)^{(q - 1)(d + 1/2)} q^{-d/2}}{(2\pi)^{(q - 1)d/2} (1 - \lVert \bb{s}\rVert_1)^{(q - 1)/2} \prod_{i\in [d]} s_i^{(q - 1)/2}} \times e^{-(q - 1)d} \bigg(\frac{q/b + qd}{q/b + d}\bigg)^{q/b + d + 1/2} \times S_{b,q}(\bb{s}).
\end{aligned}
\]

Therefore, by Stirling's approximation and the fact that $R$ is increasing on $(1,\infty)$ \citep[p.~14]{MR4319409}, we see that, as $b\to 0$,
\[
S_{b,q}(\bb{s}) = 1 + \OO_{\bb{s}}(b),
\]
and
\[
0 < S_{b,q}(\bb{s}) \leq \frac{R(q/b + d)}{R^q(1/b + d)} = 1 + \OO(b).
\]
Furthermore, a standard exponential approximation yields
\[
\begin{aligned}
\bigg(\frac{q/b + qd}{q/b + d}\bigg)^{q/b + d + 1/2}
&= \bigg\{1 + \frac{(q - 1)d}{q/b + d}\bigg\}^{q/b + d + 1/2} \\
&= e^{(q - 1)d} \{1 + \OO(b)\}.
\end{aligned}
\]
It follows from the last three equations that
\[
\begin{aligned}
A_{b,q}(\bb{s})
&= \frac{b^{-(q - 1)d/2} (1 + b d)^{(q - 1)(d + 1/2)} \{1 + \OO_{\bb{s}}(b)\}}{(2\pi)^{(q - 1)d/2} q^{d/2} (1 - \rVert \bb{s}\rVert_1)^{(q - 1)/2} \prod_{i\in [d]} s_i^{(q - 1)/2}} \\
&= \frac{b^{-(q - 1)d/2} \psi^{q - 1}(\bb{s})}{2^{-(q - 1)d/2} q^{d/2}} \{1 + \OO_{\bb{s}}(b)\}.
\end{aligned}
\]
Given that $(q - 1)/q = 1/p$, we deduce
\[
\begin{aligned}
\rVert \kappa_{\bb{s},b}\rVert_q^2
&= \bigg\{\int_{\mathcal{S}_d} \kappa_{\bb{s},b}^q(\bb{x}) \rd \bb{x}\bigg\}^{2/q} \\
&= \{A_{b,q}(\bb{s})\}^{2/q} \\
&= \frac{b^{-d/p} \psi^{2/p}(\bb{s})}{2^{-d/p} q^{d/q}} \{1 + \OO_{\bb{s}}(b)\}.
\end{aligned}
\]
This concludes the proof.
\end{proof}

The second lemma gives a uniform upper bound on the Dirichlet kernel $\bb{x}\mapsto \kappa_{\bb{s},b}(\bb{x})$ in $\mathcal{S}_d$.

\begin{lemma}\label{lem:local.bound}
Recall the definition of the function $\psi$ in Theorem~\ref{thm:MSE}. For any given $\bb{s}\in \mathcal{S}_d$, we have, as $b\to 0$,
\[
\max_{\bb{x}\in \mathcal{S}_d} \kappa_{\bb{s},b}(\bb{x}) \ll b^{-d/2} \psi(\bb{s}).
\]
\end{lemma}

\begin{proof}
See Lemma~2 of \citet{MR4319409}.
\end{proof}

%\backmatter

\section*{Reproducibility}

The \texttt{R} code used for the real-data application in Section~\ref{sec:application} can be accessed in the GitHub repository of \citet{DaayebOuimetKhardani2025GitHub}. The dataset \texttt{BDDSegX.RData} is publicly available in the GitHub repository of \citet{BarreiroLaurentThomasAgnan2021}.

\section*{Acknowledgments}
\addcontentsline{toc}{section}{Acknowledgments}

We thank Christine Thomas-Agnan for kindly providing the dataset \texttt{BDDSegX.RData} used in Section~\ref{sec:application}. We thank Christian Genest for carefully reading and editing a preliminary version of the manuscript.

\section*{Funding}
\addcontentsline{toc}{section}{Funding}

Fr\'ed\'eric Ouimet's research was partially funded through the Canada Research Chairs Program (Grant 950-231937 to Christian Genest) and the Natural Sciences and Engineering Research Council of Canada (Grant RGPIN-2024-04088 to Christian Genest). Fr\'ed\'eric Ouimet's previous postdoctoral fellowship was funded through the Natural Sciences and Engineering Research Council of Canada (Grant RGPIN-2024-05794 to Anne MacKay).

\addcontentsline{toc}{chapter}{References}

\bibliographystyle{authordate1}
\bibliography{bib}

\end{document}